\documentclass[journal,transmag]{IEEEtran}

\usepackage{hyperref,color,breqn,multirow}
\usepackage[top=0.7in, left=0.65in]{geometry}
\usepackage{color}
\usepackage{graphicx}
\hypersetup{hidelinks}
\usepackage{amssymb,amsthm}

\usepackage{booktabs}

\newtheorem{theorem}{Assertion}
\newtheorem{assumption}[theorem]{Assumption}

\theoremstyle{definition}
\newtheorem{remark}[theorem]{Remark}

\def\A{\mathbf{A}}
\def\B{\mathbf{B}}
\def\H{\mathbf{H}}
\def\J{\mathbf{J}}
\def\js{\mathbf{j}_s}
\def\hs{\mathbf{H}_s}
\def\n{\mathbf{n}}
\def\x{\mathbf{x}}
\def\v{\mathbf{v}}
\def\eps{\varepsilon}
\def\curl{\operatorname{curl}}
\def\Curl{\operatorname{Curl}}
\def\div{\operatorname{div}}
\def\Th{\mathcal{T}^h}
\def\N{\mathcal{N}}

\def\j{\mathbf{j}}
\def\f{\mathbf{f}}
\def\g{\mathbf{g}}

\def\RR{\mathbb{R}}

\begin{document}
\title{\fontsize{17pt}{17pt}\selectfont 
On the vector potential formulation with an energy-based\\ hysteresis model and its numerical solution}\author{\IEEEauthorblockN{Herbert Egger\IEEEauthorrefmark{1,2},
Felix Engertsberger\IEEEauthorrefmark{2}}
\vspace{5pt}
\IEEEauthorblockA{\IEEEauthorrefmark{1}
Johann Radon Institute for Computational and Applied Mathematics, 4040 Linz, Austria}
\IEEEauthorblockA{\IEEEauthorrefmark{2}
Institute of Numerical Mathematics, Johannes Kepler University, 4040 Linz, Austria}
}

\IEEEtitleabstractindextext{%
\begin{abstract}
The accurate modelling and simulation of electric devices involving ferromagnetic materials requires the appropriate consideration of magnetic hysteresis. We discuss the systematic incorporation of the energy-based vector hysteresis model of Henrotte et al. into vector potential formulations for the governing magnetic field equations. The field model describing a single step in a load cycle is phrased as a convex minimization problem which allows us to establish existence and uniqueness of solutions and to obtain accurate approximations by finite element discretization. Consistency of the model with the governing field equations is deduced from the first order optimality conditions. In addition, two globally convergent iterative methods are presented for the solution of the underlying minimization problems. The efficiency of the approach is illustrated by numerical tests for a typical benchmark problem.
\end{abstract}

}

\maketitle
\thispagestyle{empty}
\pagestyle{empty}

\section{Introduction}
\IEEEPARstart{M}{agnetic} hysteresis 
is one of the key properties of ferromagnetic materials and responsible for a major part of the core losses in electric motors and power transformers. It further explains the appearance of permanent magnets. The systematic incorporation of appropriate models for hysteresis into numerical solvers for magnetic field equations therefore is an important step for the simulation, optimization, and design of such devices~\cite{Meunier2008,Salon2012}.
While the physical origin of magnetic hysteresis is based on microscopic phenomena, macroscopic models like the Preisach or Jiles-Atherton model have to be used to efficiently describe the effects on the relevant scales~\cite{Jiles1986,Mayergoyz1991}.  
The use of these models in finite element simulation has been investigated intensively in the literature; see e.g.~\cite{Benabou2003,Gyselinck2004}. 
After appropriate calibration, both are capable of reproducing scalar hysteresis curves obtained in typical measurements; their extension to the vectorial case is however not trivial. Moreover, both models lack a solid thermodynamical justification.
To overcome this deficiency, Berqvist~\cite{Bergqvist1997} proposed an alternative model for magnetic hysteresis based on local energy balances and studied its numerical realization as a vector play operator. These investigations were further refined by Henrotte et al.~\cite{Henrotte2006}. 
A fully implicit time-incremental version of the \emph{energy-based magnetic hysteresis model} was presented by Francois-Lavet et al. in \cite{Lavet2013} and further analysed in \cite{Prigozhin2016,Kaltenbacher2022}.  
The energy-based model has many advantageous features: It is vectorial by construction and provides a local energy-balance and a clear definition of hysteresis losses~\cite{Bergqvist1997,Henrotte2006}. It can further be extended easily to increase accuracy and to include rate-dependent and multiphysics effects \cite{Lavet2013,Prigozhin2016}. Moreover, it can be incorporated quite naturally into magnetic field solvers based on the magnetic scalar potential~\cite{Jacques2015,Domenig2024}. 
In this paper, we study the systematic implementation of the energy-based hysteresis model in the form presented in \cite{Lavet2013} into finite element methods based on the magnetic vector potential. 
%

\subsection{Model problem}
We consider a single load step in a magneto-quasistatic field computation described by the following set of equations: 
\begin{alignat}{3}
\curl \H &= \mathbf{j_s}, \qquad & 
\div \B &= 0 \quad &&\text{in } \Omega, \label{eq:1}\\
\B = \mu_0 \H &+ \J, \qquad &  
\B \cdot \n &= 0 \quad &&\text{on } \partial\Omega. \label{eq:2}
\end{alignat}
Here $\H$ denotes the magnetic field intensity, $\B$ the magnetic induction, $\js$ the driving currents, $\mu_0$ the permeability of vacuum, and $\J$ the magnetic polarization density, which encodes the internal state of the system. 
Following \cite{Lavet2013}, this auxiliary variable is determined implicitly by solving
\begin{align} \label{eq:3}
    \min\nolimits_{\J}  \left\{ U(\J) - \langle \H ,\J \rangle + \chi |\J - \J_{p}|_\eps \right\}.
\end{align}
The internal energy density $U(\J)$ and pinning strength $\chi \ge 0$ characterize the specific material, while $\J_p$ represent the polarization from the previous load step. 
We use $|\x|_\eps = \sqrt{|\x|^2+\eps}$ to denote a regularized version of the Euclidean norm. 
\begin{remark}
For $\eps>0$, the regularized norm $|\x|_\eps$ is infinitely differentiable, which allows us to circumvent some mathematical technicalities and facilitate implementation of the model later on. 
Since $|\x|_0=|\x|$, we recover for $\eps=0$ the simplest version of the model in \cite{Lavet2013}. For general $\eps \ge 0$, one can show that $|\x| \le |\x|_\eps \le |\x| + \sqrt{\eps}$ which allows to fully control the error introduced by regularization. 
The accuracy of the model can be increased by splitting $\J = \sum_k \J^k$ into partial polarizations $\J^k$ which are updated individually by using similar minimization problems for every index $k$ separately; see~\cite{Lavet2013} and Section~\ref{sec:num}.
\end{remark}
Together with the first equation in \eqref{eq:2}, the model \eqref{eq:3} is sufficient to generate hysteresis curves with the typical characteristics~\cite{Lavet2013,Prigozhin2016}. 
The incorporation into the magnetic field problem \eqref{eq:1}--\eqref{eq:2} however leads to a fully coupled nonlinear system of differential and algebraic equations, in which the differential equations and the constitutive relation have to be treated simultaneously. 
The solution of this system is challenging, both, from a theoretical and a numerical point of view. 


\subsection{Related work}
As noted in \cite{Jacques2015,egger2024inverse}, the model \eqref{eq:3} implicitly defines a relation 
$\B = \mathcal{B}(\H;\J_p)$, called \emph{forward hysteresis operator} in the following, which describes the constitutive behaviour of the material for a given state of the internal polarization $\J_p$. 
To satisfy \eqref{eq:1}, one may then decompose $\H = \H_s - \nabla \psi$ into a known source field $\H_s$ with $\js = \curl \H_s$ and the gradient of a scalar potential $\psi$. The magnetic Gau\ss-law leads to 
\begin{align} \label{eq:4}
\div (\mathcal{B}(\H_s - \nabla \psi;\J_p)) &= 0.
\end{align}
Numerical solutions of \eqref{eq:1}--\eqref{eq:3} based on this approach were successfully demonstrated in \cite{Domenig2024,Jacques2016,Egger2024quasi}. 
Note that the mapping $\H \mapsto \mathcal{B}(\H;\J_p)$ for $\eps=0$ in \eqref{eq:3} is not differentiable in a classical sense, which complicates the analysis and numerical solution of the problem. A rigorous justification of some of the methods above was therefore only given recently in \cite{egger2025semi}. 

As shown in this reference, the mapping $\H \mapsto \mathcal{B}(\H;\J_p)$ is invertible, which allows to define the corresponding \emph{inverse hysteresis operator} $\B \mapsto \mathcal{H}(\B;\J_p))$. A compact representation of the function $\mathcal{H}(\B;\J_p))$ has been given in \cite{egger2024inverse} and its properties have been fully analysed in \cite{egger2025semi}. 
The authors of~\cite{Jacques2015} numerically inverted the material relation $\B=\mathcal{B}(\H;\J_p))$ and then considered the solution of the nonlinear system
\begin{align} \label{eq:7}
\curl (\mathcal{H}(\curl \A;\J_p)) &= \js 
\end{align}
where $\A$ is the vector potential representing $\B=\curl \A$. 
The evaluation of $\mathcal{H}(\B;\J_p)$ in this approach is rather cumbersome, as it involves the numerical minimization problem \eqref{eq:3} as well as the inversion of \eqref{eq:4}; see \cite{Jacques2018} for an extensive discussion.
The non-smoothness of the mapping $\mathcal{H}(\B;\J_p)$ causes further difficulties in the analysis and realization. 
An efficient implementation and a rigorous justification can be found in \cite{egger2025semi}.

Another quite different approach for the incorporation of the magnetic hysteresis model \eqref{eq:3} into a vector potential formulation of \eqref{eq:1}--\eqref{eq:2} was presented in the original work \cite{Lavet2013}.  
Expressing $\H = \frac{1}{\mu_0} (\B - \J)$, substituting $\B = \curl \A$, and inserting in Ampere's law gives rise to the iterative scheme
\begin{align}
&\curl (\nu_0 \curl \A^{n+1}) = \j_s + \curl(\nu_0 \J^n) \label{eq:5}\\
&\J^{n+1} = \arg\min\nolimits_\J \, \{ U(\J) - \langle \H^{n+1},\J\rangle + \chi |\J - \J_p|_\eps\} \label{eq:6}
\end{align}
with $\nu_0=1/\mu_0$ and $\H^{n+1}=\nu_0 (\curl \A^{n+1} - \J^n)$ used for abbreviation. 
It is not difficult to see that any limit of this iteration yields a solution of the original problem~\eqref{eq:1}--\eqref{eq:3}. The convergence of the scheme, however, seems unclear in general. 
As a side result, we will present a slight modification of this method which can be shown to converge globally.

\subsection{Main contributions and outline.}
In this paper, we present a novel systematic approach to the numerical solution of \eqref{eq:1}--\eqref{eq:3} by means of a vector potential formulation. 
For ease of notation, all results are presented for problems in three space dimensions and the magnetic hysteresis model \eqref{eq:3} with a single internal variable. The modifications required for the two-dimensional case and multiple pinning forces will be outlined further below.
In Section~\ref{sec:variational}, we derive a variational characterization of solutions to \eqref{eq:1}--\eqref{eq:3} in terms of the magnetic vector potential $\A$ and magnetic polarization $\J$, and briefly comment on its well-posedness. 
In Section~\ref{sec:solution} we discuss the systematic discretization of the problem by finite elements and present two globally convergent iterative solvers. 
Numerical tests are presented in Section~\ref{sec:num} for illustration of the performance of the proposed algorithms and a comparison with alternative approaches. 
The presentation closes with a short summary of our results and topics for future research.

\section{Variational characterization of solutions}
\label{sec:variational}

We start with laying the theoretical foundations for our further investigations. The key step is to identity the solutions of \eqref{eq:1}--\eqref{eq:3} as those of an equivalent variational problem.

\subsection{Notation and main assumptions}
Let $\Omega \subset \RR^3$ be a bounded domain. By $L^2(\Omega)$, $H^1_0(\Omega)$, and $H_0(\curl;\Omega)$ we denote the usual spaces of square integrable functions and those with well-defined gradient or curl and vanishing traces at the boundary~\cite{Monk2003}.
We may decompose
\begin{align} \label{eq:split}
H_0(\curl) = \nabla H_0^1(\Omega) \oplus V
\end{align}
into a direct sum of closed subspaces, such that $\|\curl \cdot\|_{L^2(\Omega)}$ defines a norm on $V$ under some topological restrictions on the domain~\cite{Monk2003}. We further abbreviate $Q=L^2(\Omega)^3$ in the sequel.
For our analysis, we make use of the following assumptions, which allow to us to state and prove rigorous theoretical results. 

\begin{assumption} \label{ass:1}
$\Omega \subset \RR^3$ is a bounded Lipschitz domain and topologically trivial. 
The space $V$ in \eqref{eq:split} is chosen such that the splitting is direct and stable. 
The energy density is defined as
$U(\J)=-\frac{2 A_s J_s}{\pi} \log(\cos(\frac{\pi}{2}\frac{\J}{J_s}))$ with $A_s$, $J_s>0$, and we set $\chi \ge 0$. Finally, 
$\js =\curl \hs$ for some $\hs \in H(\curl;\Omega)$.
\end{assumption}
Various generalizations of these conditions are possible, e.g. the incorporation of spatial dependence of $A_s$, $J_s$, and $\chi$, different forms of the energy densities $U(\J)$, or the consideration of multiple pinning forces; see \cite{Lavet2013,Egger2024quasi} and Section~\ref{sec:num}. 

\subsection{A variational problem}

We will show below that the solution of \eqref{eq:1}--\eqref{eq:3} can be constructed from the minimizer of the variational problem
\begin{align} \label{eq:min}
\min_{\A \in V} \min_{\J \in Q} \int_\Omega \frac{\nu_0}{2} |\curl \A &- \J|^2 - \langle \H_s, \curl \A\rangle \\
&+ U(\J) + \chi |\J-\J_p|_\eps \, dx. \notag
\end{align}
Note that the function spaces $V \subset H_0(\curl;\Omega)$, defined above, already incorporates the relevant gauging and boundary conditions for the vector potential $\A$. 
As a first step in our analysis, let us clarify the well-posedness of this problem.
\begin{theorem}
Let Assumption~\ref{ass:1} hold. Then for any $\eps \ge 0$ and any $\J_p \in Q$, problem \eqref{eq:min} admits a unique minimizer.
\end{theorem}
\begin{proof}
Using $\B = \curl \A$, the integrand in \eqref{eq:min} amounts to
$$
\frac{\nu_0}{2} |\B - \J|^2 - \langle \H_s, \B \rangle + U(\J) + \chi |\J-\J_p|_\eps.
$$
This function is strongly convex, coercive, and continuous on its domain. 
Existence of a unique solution then follows from standard arguments of convex analysis~\cite{Ekeland1999}. 
\end{proof}

\subsection{First order optimality conditions}
Since the functional to be minimized in \eqref{eq:min} is convex, its minimizers are equivalently characterized by the first order optimality conditions. For parameter $\eps>0$, they read
\begin{alignat*}{2}
\int_\Omega \nu_0 (\curl \A - \J) \cdot \curl \A' - \H_s \cdot \curl \A' &= 0, \\ 
\int_\Omega (\nu_0(\J - \curl \A) + \nabla U(\J) + \chi \frac{\J-\J_p}{|\J-\J_p|_\eps}) \cdot \J' &= 0, 
\end{alignat*}
which hold for all test functions $\A' \in V$ and $\J' \in Q$.
The strong form of these identities amounts to the equations
\begin{alignat}{2} 
&\curl (\nu_0 (\curl \A - \J)) = \j_s \quad &&\text{in } \Omega,\label{eq:opt1}\\
&\nabla U(\J) -\H  +  \chi \frac{\J-\J_p}{|\J-\J_p|_\eps} = 0\quad && \text{in } \Omega, \label{eq:opt2}
\end{alignat}
where $\curl \H_s=\j_s$ and $\H=\nu_0(\B-\J)$ with $\B=\curl \A$ was used to rewrite the two equations in a convenient form. 
These observations immediately lead to the following result.
\begin{theorem} \label{thm:4}
Let Assumption~\ref{ass:1} hold and $\eps >0$. Then the minimizer of problem \eqref{eq:min} is characterized by \eqref{eq:opt1}--\eqref{eq:opt2} together with the corresponding boundary and gauging conditions contained in the definition of the space $V$.
\end{theorem}
For completeness of the presentation, let us briefly comment on the extension of the result to the limiting case $\eps=0$.
\begin{remark}
For $\eps=0$, the function $|\J-\J_p|_\eps = |\J - \J_p|$ is not differentiable at $\J=\J_p$. In this case, equation \eqref{eq:opt2} has to be replaced by the inclusion
\begin{alignat}{2} \label{eq:opt2b}
&\H - \nabla U(\J) \in  \chi \partial |\J-\J_p|_\eps.
\end{alignat}
Here $\partial f(x)$ denotes the sub-differential of the convex function $f$ at $x$, i.e., the set of all vectors (or corresponding linear functions) describing the supporting hyperplanes to the graph of $f$ at $x$; see \cite{Rockafellar1970} for details.
For $f(\J)=|\J-\J_p|$, we have 
\begin{align*}
\partial |\J-\J_p| 
=\begin{cases}
\{\frac{\J-\J_p}{|\J-\J_p|}\}, & \J \ne \J_p, \\ 
\{\x : |\x| \le 1\}, & \J = \J_p.
\end{cases}
\end{align*}
In the smooth case $\eps>0$, we have $\partial |\J - \J_p|_\eps=\{\frac{\J-\J_p}{|\J-\J_p|_\eps}\}$ for any $\J$. 
Then \eqref{eq:opt2b} reduces to \eqref{eq:opt2}. 
Assertion~\ref{thm:4} thus remains valid for all $\eps \ge 0$, 
if we replace \eqref{eq:opt2} by \eqref{eq:opt2b}. 
This allows to treat all cases in a unified manner.
\end{remark}

\subsection{Equivalence with the model problem}
With these preliminary considerations, we can now characterize solutions of problem~\eqref{eq:1}--\eqref{eq:3} via the minimizers of \eqref{eq:min}. 

\begin{theorem}
Let Assumption~\ref{ass:1} hold and $(\A,\J)$ denote the unique minimizer of problem \eqref{eq:min} with $\eps \ge 0$. Then $\J$, together with $\B=\curl \A$ and $\H=\nu_0(\B - \J)$, satisfies \eqref{eq:1}--\eqref{eq:3}. 
\end{theorem}
\begin{proof}
Substituting $\H=\nu_0(\B - \J) = \nu_0 (\curl \A - \J)$ in \eqref{eq:opt1} yields the first equation in \eqref{eq:1}. The ansatz $\B=\curl \A$ and the boundary conditions of $\A$ in the definition of the set $V$ further yield $\div \B=0$ and $\B \cdot \n=0$, which already shows validity of \eqref{eq:1}--\eqref{eq:2}. 
Relation \eqref{eq:opt2} respectively \eqref{eq:opt2b} finally characterize the minimizers of the convex minimization problem \eqref{eq:3}. 
\end{proof}

\section{Numerical solution}
\label{sec:solution}

The results of the previous section show that solutions of our model problem \eqref{eq:1}--\eqref{eq:3} can be characterized via the minimizers of the variational problem~\eqref{eq:min}.  
We now show that this problem is also very well-suited for a systematic numerical treatment. 
We first briefly discuss the spatial discretization by finite elements and then present two iterative solution strategies.

\subsection{Finite element discretization}
Let $\Th = \{T\}$ denote a tetrahedral mesh of the domain $\Omega$. 
We write $P_k(\Th)=\{v : v|_T \in P_k(T) \ \forall T \in \Th\}$ for the spaces of piecewise polynomials of order $k$. We write 
$\tilde V_h = \{\v_h \in H_0(\curl;\Omega) : \v_h |_T \in \N_0(T) \ \forall T \in \Th\}$
for the lowest order N\'ed\'elec finite element space with zero boundary conditions~\cite{Nedelec1980,Boffi2013}. This space admits a stable splitting 
\begin{align} \label{eq:split_h}
\tilde V_h = \nabla S_h \oplus V_h
\end{align}
where $S_h = \{s_h \in H_0^1(\Omega) : s_h |_T \in P_1(T) \ \forall T \in \Th\}$ is the standard piecewise linear finite element space~\cite{Monk2003,Boffi2013}. A basis for the space $V_h$ can be obtained, e.g. by tree-cotree splitting~\cite{Albanese1988,Rapetti2022}. 
We further define $Q_h = P_0(\Th)^3$ to denote the space of piecewise constant vector functions. 

For discretization of our model problem, we then simply use the approximation of the minimization problem \eqref{eq:min} in the appropriate finite element spaces. This amounts to solving
\begin{align} \label{eq:min_h}
\min_{\A_h \in V_h} \min_{\J_h \in Q_h} \int_\Omega \frac{\nu_0}{2} |\curl \A_h &- \J_h|^2 - \langle \j_s, \A_h\rangle \\
&+ U(\J_h) + \chi |\J_h-\J_p|_\eps \, dx. \notag
\end{align}
As immediate consequence of the problem structure, we can establish the well-posedness of this discretization strategy.
\begin{theorem}
Let Assumption~\ref{ass:1} hold and $V_h$, $Q_h$ be defined as described above. Then for any $\eps \ge 0$ and $\J_p \in Q_h$, problem~\eqref{eq:min_h} admits a unique solution.     
\end{theorem}
\begin{proof}
Formula~\eqref{eq:min_h} amounts to a strongly convex minimization problem over the finite dimensional spaces $V_h$, $Q_h$. Existence of a unique solution thus follows immediately. 
\end{proof}

\begin{remark}
For ease of presentation, we here only consider the simplest case of a possible approximation scheme. Higher order approximations, curved meshes, numerical quadrature, etc. can be considered with very similar arguments~\cite{Egger2024ho}. As shown in this reference, the variational structure of the problem is also beneficial for a systematic error analysis and the design and the convergence analysis of iterative solvers.
\end{remark}

We now turn to the iterative solution of \eqref{eq:min} and \eqref{eq:min_h}. 
For ease of presentation, we discuss the methods on the continuous level. Since the structure of the minimization problem is preserved by discretization, all results also hold for the finite element approximations.
Moreover, the used arguments show that convergence of the proposed methods will hold independent of the mesh size.

\subsection{A regularized Newton method}

We assume $\eps>0$, such that the integrand in \eqref{eq:min} is strongly convex and at least two-times differentiable. For minimization, we can then employ the Newton method with line search~\cite{Nocedal2006}. 
For given iterate $(\A^n,\J^n)$, we compute the Newton update direction $(\delta \A^n,\delta \J^n)$ by solving the linearized system 
\begin{align}
\curl (\nu_0 (\curl \delta \A^n - \delta \J^n)) &= \f^n, \label{eq:new1} \\
- \nu_0 \curl \delta \A^n + \nu^n \delta \J^n &= \g^n. \label{eq:new2}
\end{align}
The gauging and boundary conditions for the vector potential $\A$ are included automatically in the weak form of the problem.
The right hand sides in these equations amount to the residuals in the optimality system \eqref{eq:opt1} and \eqref{eq:opt2}, and hence read
\begin{align}
\f^n &= \j_s - \curl(\nu_0 (\curl \A^n - \J^n)) \label{eq:new3} \\
\g^n &= \nu_0 (\curl \A^n - \J^n) - \nabla_\J U(\J) - \chi \tfrac{\J^n - \J_p}{|\J^n - \J_p|_\eps}. \label{eq:new4}
\end{align}
The reluctivity tensor $\nu^n$ in \eqref{eq:new2} stems from differentiating \eqref{eq:opt2} with respect to $\J$, and is thus given by 
\begin{align} \label{eq:new5}
\nu^n = \nu_0 \text{I} + \nabla_{\J\J} U(\J) + 
\frac{\chi}{|\J-\J_p|_\eps} (I - \v \otimes \v), 
\end{align}
with scaled vector $\v = \frac{\J - \J_p}{|\J-\J_p|_\eps}$ used for abbreviation. 
The next Newton-iterate is finally defined by the update rule
\begin{align} \label{eq:new6}
\A^{n+1} = \A^n + \tau^n \delta \A^n, \quad 
\J^{n+1} = \J^n + \tau^n \delta \J^n
\end{align}
with appropriate step size $\tau^n>0$. In our numerical tests, we employ \emph{Armijo back-tracking} with parameters $q=0.5$ and $\sigma=0.1$, which defines the step size $\tau^n$ according to \cite{Nocedal2006}
\begin{align} \label{eq:new7}
\tau^n = \max\{ \tau = q^n : \Phi(\tau) \le \Phi(0) + \sigma \tau \Phi'(0), \ n \ge 0\}.
\end{align}
Here 
$\Phi(\tau)$ denotes the value of the functional to be minimized in \eqref{eq:min} at $\A=\A^{n} + \tau \delta \A^n$ and $\J=\J^n + \tau \delta \J^n$.
Convergence of the iteration \eqref{eq:new1}--\eqref{eq:new7} follows from standard arguments. 
\begin{theorem}
Let Assumption~\ref{ass:1} hold and $\eps > 0$. Then for any $\A^0 \in V$, $\J^0,\J_p \in Q$, the Newton-method with Armijo line-search defines a sequence of iterates $\A^n \in V$, $\J^n \in Q$, $n \ge 1$, which converges at least linearly to the minimizer of \eqref{eq:min}. 
\end{theorem}
\begin{proof}
The discrete setting is fully covered by the results in \cite{Nocedal2006}. The arguments presented in \cite{Heid2023,egger2024global} allow to treat also the infinite dimensional problem in a rigorous manner.
\end{proof}
\begin{remark} \label{rem:reduction}
Equation \eqref{eq:new2} is local in space which allows to compute $\delta \J^n = (\nu^n)^{-1} (\g^n + \nu_0 \curl \delta \A^n)$ explicitly. Inserting this into \eqref{eq:new1} leads to a single linear magnetic field equation 
\begin{align} \label{eq:red}
\curl(\tilde \nu^n \delta \A^n) &= \tilde \f^n
\end{align}
with reluctivity tensor $\tilde \nu^n = \nu_0 (I - (\nu^n)^{-1} \nu_0)$ and right hand side $\tilde \f^n=\f^n + \nu_0 (\nu^n)^{-1} \g^n$. 
The increments $\delta \A^n$ of the magnetic vector potential are thus characterized by standard linear magnetostatic problem which have the same form as those arising in nonlinear magnetic field computations without hysteresis. 
The main computational effort required to realize one Newton-step in the above iteration, i.e., the computation of $\delta \A^n$, is thus not affected by the presence of hysteresis.
\end{remark}

\subsection{Block-coordinate descent}
\label{sec:fixpoint}
On the abstract level, the two minimization problems \eqref{eq:min} and \eqref{eq:min_h} have the common abstract form 
\begin{align} \label{eq:min2}
\min_\A \min_\J f(\A,\J),
\end{align}
with $f$ denoting an appropriate strongly convex functional. 
A well-known strategy for solving such problems is to minimize separately with respect to $\A$ and $\J$ in alternating order. 
This amounts to first updating
$ 
\A^{n+1} = \arg\min\nolimits_\A f(\A,\J^{n})$ and then $\J^{n+1} = \arg\min\nolimits_\J f(\A^{n+1},\J)$. Convergence of such schemes can be established in rather general situations~\cite{Carstensen1997}.

For the problem under consideration, the block-coordinate descent algorithm leads to the update equations
\begin{align}
&\curl (\nu_0 \curl \A^{n+1}) = \j_s + \curl (\nu_0 \J^{n}), \label{eq:bcd1}\\
&\min\nolimits_{\J} U(\J) + \frac{\nu_0}{2} |\curl \A^{n+1} - \J|^2 + \chi |\J - \J_p|_\eps. \label{eq:bcd2}
\end{align}
Similar to before, we search for weak solutions $\A^{n+1} \in V$ and $\J^{n+1} \in Q$ for these equations, and thus automatically take account of the corresponding boundary and gauging conditions. 
The particular structure of the underlying variational problem allows to establish the following convergence result. 
\begin{theorem} \label{thm:bcd}
Let Assumption~\ref{ass:1} hold and $\eps \ge 0$. Then for any $\A^0 \in V$ and $\J^0 \in Q$, the weak solutions of \eqref{eq:bcd1}--\eqref{eq:bcd2}, define a sequence $\A^n \in V$ and $\J^n \in Q$ for $n \ge 1$.
Moreover, the iterates converge linearly to the unique solution of \eqref{eq:min}. 
\end{theorem}
\begin{proof}
The results follow immediately from those in~\cite{Carstensen1997}.
\end{proof}

\begin{remark}
The interpretation of \eqref{eq:bcd1}--\eqref{eq:bcd2} as a block-coordinate descent algorithm for the convex minimization problem \eqref{eq:min} here allows for a rigorous convergence analysis. The iteration \eqref{eq:5}--\eqref{eq:6}, mentioned in the introduction, has a similar structure, but the function to be minimized in the second step is slightly different. This method hence does not make use of the underlying variational problem and its convergence behaviour, therefore, seems at least unclear. 
\end{remark}

\section{Numerical illustration}
\label{sec:num}

To demonstrate the feasibility and efficiency of the proposed methods, we now report about the simulation of the magnetic fields in the T-joint of a three-limb transformer. The test problem under consideration is motivated by~\cite{Gyselinck2004,Jacques2015}. 

\subsection{Model problem}
The specific setup of the problem allows to perform the computations on a two-dimensional cross-section $\Omega \subset \RR^2$ of the device which is depicted in Figure \ref{fig:geometry:tjoint}. 
\begin{figure}[!ht]
    \centering
    \includegraphics[trim = {1cm 1cm 0cm 0.5cm},clip,width=0.7\linewidth]{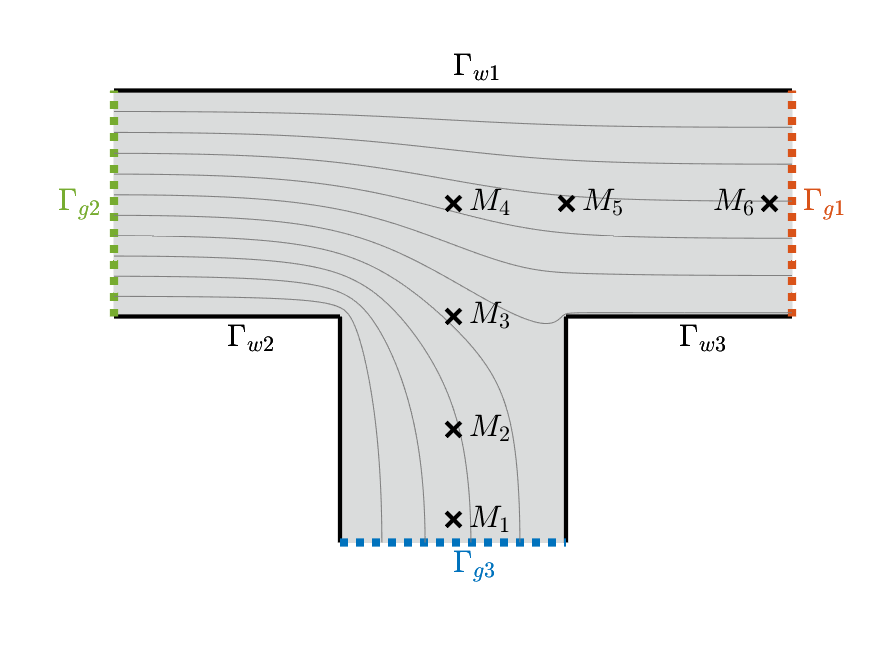}
    \vspace{-0.4cm}
    \caption{2D cross-section of the T-joint of a transformer taken from \cite{Gyselinck2004,Jacques2015}. On the flux walls $\Gamma_w$ (black) the boundary condition $\B \cdot \mathbf{n} = 0$ is imposed, while on the flux gates $\Gamma_g$ (green, red, blue) we prescribe $\H \times \mathbf{n} = 0$. In addition, a time-varying flux $\phi$ is imposed at these gates. 
    Field values at the measurement points $M_1$--$M_6$ will be reported below.}
    \label{fig:geometry:tjoint}
\end{figure}
The vector potential then has the form $\A=(0,0,A_z)$ and the magnetic field and flux are given by $\H=(H_x,H_y,0)$ and $\B=(B_x,B_y,0)$. All field components are independent of the $z$-coordinate. 

On the flux walls $\Gamma_{w\ell}$ of the geometry, the boundary condition $\B \cdot \n = 0$ is imposed, while on the flux gates $\Gamma_{g\ell}$, we require $\H \times \n = 0$ and additionally prescribe the total flux $\Phi_\ell = \int_{\Gamma_{g\ell}} \B \cdot \n \, ds$. The condition
$\sum_\ell \Phi_\ell=0$ is required to stay compatible with the magnetic Gau\ss' law $\div \B=0$. 

The energy-based magnetic hysteresis model with $5$ pinning forces presented in \cite{Lavet2013} is used to describe the ferromagnetic material behaviour. 
The total polarization $\J=\sum_k \J_k$ thus is split into $5$ partial polarizations whose updates are defined by 
\begin{align*}
\min_{\J_k} U_k(\J_k) - \langle \H,\J_k\rangle + \chi_k |\J_k - \J_{k,p}|_\eps
\end{align*}
The functions $U_k(\J_k) = - \frac{A_s J_{s,k}}{\pi} \log \left(\cos \left(\frac{\pi}{2} \frac{|\J_k|}{J_{s,k}}\right)\right)$ are used as internal energy densities. The parameters in the model used in our simulations are summarized in the following table.
\begin{table}[h!]
\centering
\small
\caption{Material data for 5-cell model from 
\cite{Lavet2013}.}
\label{tab:material:lavet}
\vspace{-0.2cm}
\begin{tabular}{@{}cll@{}}
\toprule 
\textbf{parameter} & \textbf{values} & \textbf{description}\\
\midrule
$A_s$ & 65 & field strength \\
$J_{s,k}$  & 0.11, 0.3, 0.44, 0.33, 0.04 & reference polarization\\
$\chi_k$  & 0, 10, 20, 40, 60 & pinning strength\\
\bottomrule
\end{tabular}
\end{table}
\vspace{-0.2cm}
\subsection{Vector Potential Formulation} 
To incorporate the specific boundary conditions mentioned above, we split the relevant out of plane component of the magnetic vector potential as 
\begin{align*}
A_z = A_{z,g} + A_{z,0}
\end{align*}
with $A_{z,0} \in W = \{w \in H^1(\Omega) : w = 0 \text{ on } \Gamma_{w\ell}\}$ and remainder $A_{z,g} \in \{w \in H^1(\Omega) : w = c_\ell \text{ on } \Gamma_{w\ell} = c_\ell, \ c_\ell \in \RR \}$. 
The in-plane components of the magnetic flux are then defined by 
\begin{align*}
\B = \binom{B_x}{B_y} = \Curl A_z = \binom{\partial_y A_z}{-\partial_x A_z}.
\end{align*}
Here $\Curl$ is the scalar-to-vector curl in two dimensions. 
The function $A_{z,g}$ can be constructed such that $\div \B=0$ and the required boundary conditions  hold for any choice of $A_{z,0}$ in the  decomposition~\eqref{eq:split}; see~\cite{Jacques2015,egger2025semi}. 
The variational formulation of the corresponding magnetic field problem then reads
\begin{align} \label{eq:mink}
&\min_{A_{z,0} \in W} \min_{\J_k \in Q} \int_\Omega \frac{\nu_0}{2} |\Curl (A_{z,0} + A_{z,g}) - \sum\nolimits_k \J_k|^2 \\
&\qquad \qquad \qquad \qquad + \sum\nolimits_k U_k(\J_k) + \chi_k |\J_k - \J_{k,p}|_\eps \, dx. \notag
\end{align}
All results presented in the previous section generalize immediately to this problem. In particular, problem \eqref{eq:mink} admits a unique solution and $\B=\Curl \A_z$, $\H=\nu_0 (\B - \sum_j \J_k)$ satisfy $\div \B=0$ and $\curl \H=0$ as well as the mixed boundary conditions stated above. 
Moreover, the problem is suitable for discretization by finite elements and iterative solution by the regularized Newton or block-coordinate descent methods. 

\subsection{Details on the implementation}
We first describe the finite element approximation used in our numerical tests. To do so, let $\Th$ be a regular triangulation of the domain $\Omega$ and let 
$$
W^h = W \cap P_1(\Th)
$$
denote the space of piecewise linear functions over $\Omega$ which are continuous across element interfaces and vanish on the flux walls $\Gamma_{w\ell}$. The partial polarizations are approximated in the spaces $Q^h = P_0(\Th)^2$ of piecewise constant vector fields. 
The discrete problem amounts to minimizing \eqref{eq:mink} over $A_{z,0}^h \in W_h$ and $\J_k^h \in Q^h$. Existence of a unique solution again follows immediately from convexity of the problem. 

Application of the regularized Newton method and the block coordinate-descent method outlined in the previous section is now straight forward. Since the space $Q_h$ used for the approximations $\J_k^h$ of the partial polarizations consists of piecewise constant vector fields, the elimination procedure described in Remark~\ref{rem:reduction} can be carried out verbatim on the discrete level. Also the minimization problems for computing $(\J_k^h)^{n+1}$ in the block-coordinate descent method, see \eqref{eq:bcd2}, decouple and can be solved separately for every element of the mesh $\Th$. 
As a consequence, the main numerical effort for performing one iteration step in the two methods consists of solving the discretized version of a linearized magneto-static field problem of the form \eqref{eq:red} respectively \eqref{eq:bcd1}. 

All implementations were done in \textsc{Matlab} and performed on a laptop equipped with a 13th Gen Intel Core i5-1335U CPU with clock rate 1.30 GHz. The built-in sparse direct solvers are used for solution of the linear systems arising from discretization of \eqref{eq:red} and \eqref{eq:bcd1}. 
Iterations are stopped, when the difference in the function value is smaller than $10^{-6} f_0$ with $f_0$ the function value at the start of the iteration.

\subsection{Numerical results}

We simulate a typical load cycle starting from a completely demagnetized state with initial polarizations $\J_{k}(0) = 0$.
The fluxes at the gates $\Gamma_{g\ell}$ are chosen as $\Phi_\ell(t) = \cos(2\pi t + 2\pi\ell/3) \cdot f(t)$ with amplitude factor $f(t)=(1- \cos(4 \pi t ))/2$ for $t \le 1/4$ and $f(t)=1$ for $t \ge 1/4$;  see \cite{Gyselinck2004}. Note that $\sum_\ell \Phi_\ell=0$ for all $t$ by construction.
We partition the time interval $[0,2]$ into 200 equidistant time-steps $t^n = n \tau$ with $\tau = 0.02$ and determine $A_{z,g}$ corresponding to the fluxes $\Phi_\ell(t^n)$, as described above. We further update $\J_{k,p}(t^{n}) = \J_k(t^n)$ from the previous time step.
The values $A_{z,0}^h(t^n)$, $\J_k^h(t^n)$ for the new time step $t^n$, $n \ge 1$, are then computed by solving \eqref{eq:mink} numerically. 

The results of our computations are summarized in Table~\ref{tab:1}. As expected both methods converge globally and mesh-independent with constant average iteration counts. The Newton method outperforms the block coordinate descent method significantly. We also showcase the resulting iron losses, since the power loss density in the current time-step is directly available as $\frac{1}{\tau}\sum_k \chi_k |\J_k - \J_{k,p}|$, for which we assumed a transformer depth of $1$m.
\begin{table}[ht!]
\centering
\small
\caption{Total computation times and average iteration numbers for the Newton Method and Block Coordinate Descent on different mesh refinement levels and 200 time-steps. }
\begin{tabular}{c || cc | cc | c }
\toprule 
\textbf{dof} 
& \multicolumn{2}{c}{\textbf{Newton Method}} 
& \multicolumn{2}{c}{\textbf{Block Coordinate}} 
& \textbf{Iron Losses}   \\
\midrule
$570$    & $9.92$ s  & $8.40$ it & $812$ s& $722$ it& $698.1$ J \\
$2170$ & $34.3$ s & $8.25$ it & $2930$ s & $723$ it & $697.4$ J  \\
$8463$ & $155$ s & $8.32$ it & $12.34$k s & $724$ it & $697.0$ J \\
$33421$ & $1029$ s & $8.06$ it & $52.18$k s & $727$ it & $696.8$ J \\
\bottomrule
\end{tabular}
\label{tab:1}
\end{table}
The locus curves for steady state operation in the time interval $[1,2]$ are illustrated in Figure \ref{fig:hxhy}. Clearly, the typical hysteretic behaviour and saturation effect are captured in the magnetic field simulation.
\begin{figure}[!ht]
    \centering
    \includegraphics[trim = {1.7cm 0cm 2.5cm 0.8cm},clip,width=0.47\linewidth]{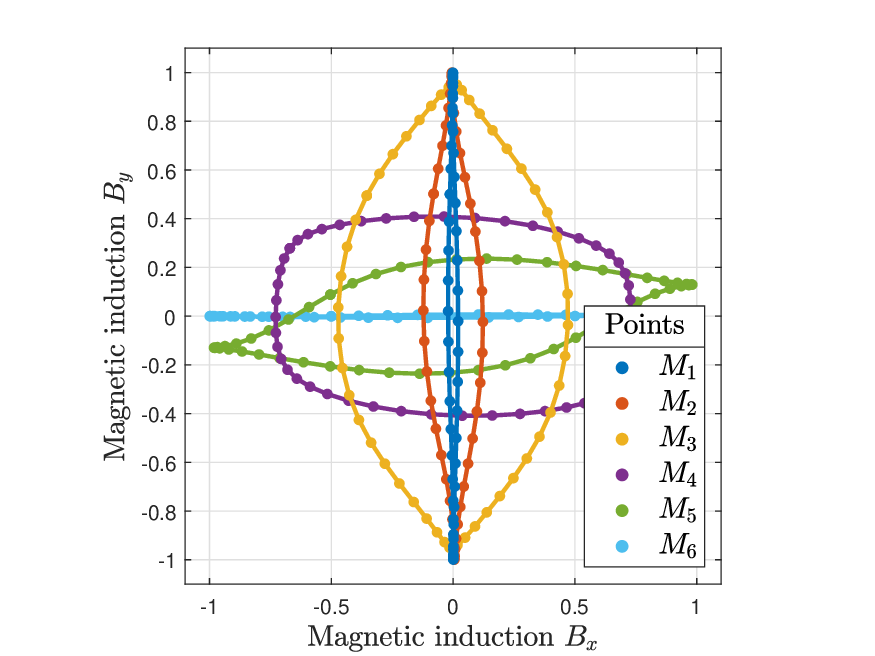}
    \hspace{0.1cm}
    \includegraphics[trim = {1.7cm 0cm 2.5cm 0.8cm},clip,width=0.47\linewidth]{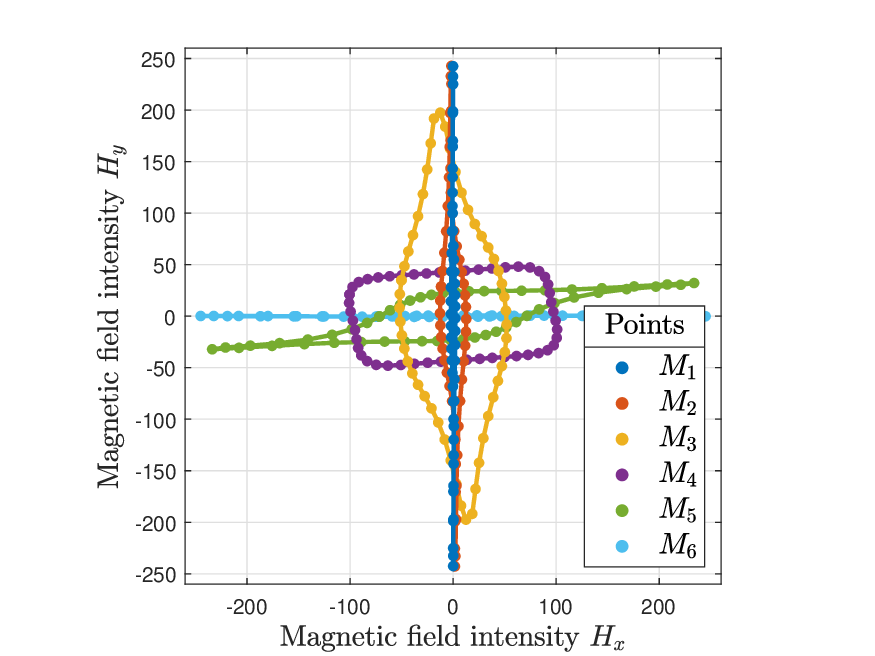}
    \caption{Locus curves of the of the magnetic induction $\B$ (left) and magnetic
    field intensity $\H$ (right) at different measurement points M (see Figure \ref{fig:geometry:tjoint}).}
    \label{fig:hxhy}
\end{figure}
\section{Discussion}
\label{sec:discussino}
In this paper, we discussed a systematic way of incorporating the energy-based hysteresis model of Henrotte et al. \cite{Henrotte2006,Lavet2013} into magnetic vector potential formulations. The key step was to realize that the corresponding quasi-static fields can be characterized via the minimizers of a strongly convex variational problem which involves the simultaneous minimization over the vector potential and the magnetic polarizations. 
Regularization of the non-smooth terms in the hysteresis model allowed us to circumvent technicalities and to apply standard Newton-methods with line search for the iterative solution of the underlying nonlinear field problems. 
The particular problem structure further allowed us to establish well-posedness of the problem and to consider its numerical solution by finite element discretization and globally convergent iterative solvers. Some details about the efficient implementation were highlighted and the performance of the proposed algorithms was demonstrated by numerical tests. 
Further study of the limiting case $\eps=0$ and its numerical solution by generalized Newton-methods~\cite{Gfrerer2021} are left as topics for future research.
\section*{Acknowledgements}
This work was supported by the joint DFG/FWF Collaborative Research Centre CREATOR (DFG: Project-ID 492661287/TRR 361; FWF: 10.55776/F90).


\end{document}